\documentclass[10pt]{article}

\usepackage{amsmath}
\usepackage{amsfonts}
\usepackage{amsthm}
\usepackage[english]{babel}
\usepackage{graphicx}
\usepackage[all]{xy}
\usepackage{amssymb}
\usepackage{lineno}

\newtheorem{theorem}{Theorem}
\newtheorem{lemma}{Lemma}

\newtheorem{remark}{Remark}

\begin{document}
\title{The kernel of the  generalized Clifford-Fourier transform and its generating function}

\author{Pan Lian$^{1,2}$\footnote{E-mail: {\tt pan.lian@ugent.be (corresponding author)}} \and Gejun Bao$^{1}$ \footnote{E-mail: {\tt baogj@hit.edu.cn}} \and Hendrik De Bie$^{2}$\footnote{E-mail: {\tt hendrik.debie@ugent.be}}\and Denis Constales$^{2}$ \footnote{E-mail: {\tt  denis.constales@ugent.be}}  }

\vspace{10mm}
\date{\small{1: Department of Mathematics - Harbin Institute of Technology\\
West Da-Zhi Street 92, 150001 Harbin, P.R.China}\\ \vspace{5mm}
 \small{2:  Department of Mathematical Analysis\\ Faculty of Engineering and Architecture -- Ghent University}\\
\small{Galglaan 2, 9000 Gent,
Belgium}\\\vspace{5mm}}



\maketitle

\begin{abstract}

 In this paper, we study the generalized Clifford-Fourier transform introduced in \cite{Ro1} using the Laplace transform technique. We give explicit expressions in the even dimensional case, we obtain polynomial bounds for the kernel functions and establish a generating function.

~\\


\end{abstract}

\tableofcontents

\section{Introduction}

In recent years, quite some attention has been devoted to the study of the so-called Clifford-Fourier transform. This transform, first established in \cite{FNF, FNF2} is a genuinely non-scalar generalization of the Fourier transform, developed within the framework of Clifford analysis \cite{ca}. Indeed, it can be written as

\[
F_{-}(f)(y)=(2\pi)^{-m/2}\int_{\mathbb{R}^{m}} K_m(x,y)f(x)dx
\]
with
\[
K_{m}(x,y)=e^{ i\frac{\pi}{2}\Gamma_{y}}e^{-i(x,y)}
\]
with $\Gamma_y$ the spherical Dirac operator (see equation (\ref{ga})).

It turned out to be a difficult problem to determine the kernel $K_{m}(x,y)$ explicitly. This was first achieved in \cite{HY} using plane wave decompositions. Later, in \cite{MH} a different method using wave equations was established. In \cite{CDL}, a short proof was obtained by considering the Clifford-Fourier kernel in the Laplace domain, where it takes on a much simpler form.

 Our aim in the present paper is to develop the Laplace transform method for a much wider class of generalized Fourier transforms. According to investigations in \cite{Ro1} using the representation theory for the Lie superalgebra $\mathfrak{osp}(1|2)$, the following expression
\begin{equation}
\label{Kexpr}
e^{i \frac{\pi}{2} G(\Gamma_y) }e^{-i(x,y)}
\end{equation}
where $G$ is an integer-valued polynomial can be used as the kernel for a generalized Fourier transform that still satisfies properties very close to that of the classical transform.
The extension of the Laplace transform technique to kernels of type (\ref{Kexpr}) will allow us to find explicit expressions for the kernel. We will moreover determine which polynomials $G$ give rise to polynomially bounded kernels and we will determine the generating function corresponding to a fixed polynomial $G$.

The paper is organized as follows. In order to make the exposition self-contained, in Section 2,  we recall basic facts of the Laplace transform, Clifford analysis and the generalized Clifford-Fourier transform. Section 3 is devoted to establishing the connection between the kernel of the fractional Clifford-Fourier transform \cite{CDL} and the generalized Clifford-Fourier transform. We first compute a special case in Section 3.1. Then the method is generalized to the case in which the polynomial has integer coefficients in Section 3.2. The kernel and the generating function in  the even dimensional case are given. We also discuss  which kernels are polynomially bounded.

\section{Preliminaries}
\subsection{The Laplace transform}
 The Laplace transform of a real or complex valued function $f$ which has exponential order $\alpha$, i.e. $|f(t)|\le Ce^{\alpha t}, t\ge t_{0}$ is defined as
\[F(s)=\mathcal{L}(f(t))=\int_{0}^{\infty}e^{-st}f(t)dt.\]  By Lerch's theorem \cite{Lp1}, the inverse transform \[\mathcal{L}^{-1}(F(s))=f(t)\] is uniquely defined when we restrict to functions which are continuous on $[0,\infty)$.  Usually, we can use integral transform tables (see e.g. \cite{lp2}) and the partial fraction expansion to compute the Laplace transform and its inverse. We list some  which will be used in this paper:
\begin{eqnarray}
\label{l7}\mathcal{L}(e^{-\alpha t})&=&\frac{1}{s+\alpha};\\
\label{l8} \mathcal{L}(t^{k-1}e^{-\alpha t})&=&\frac{\Gamma(k)}{(s+\alpha)^{k}},   \quad k>0.
\end{eqnarray}
We also need the convolution formula and the inverse Laplace transform. Denote by $r=(s^{2}+a^{2})^{1/2}$, $R=s+r$, $G(s)=\mathcal{L}(g(t))$ and $F(s)=\mathcal{L}(f(t))$. We have
\begin{eqnarray}\label{cv1} G(s)F(s)&=&\mathcal{L}(\int_{0}^{t}g(t-\tau)f(\tau)d\tau);\\
\label{l2}\mathcal{L}^{-1}(a^{\nu}r^{-2\nu-1})&=&2^{\nu}\pi^{-1/2}\Gamma(\nu+\frac{1}{2})t^{\nu}J_{\nu}(at),  \qquad   \mbox{Re}(\nu) >-1/2,  \mbox{Re} (s)>|\mbox{Im} (a)|. \end{eqnarray}


\subsection{Clifford analysis and generalized Fourier transforms }
In this section, we give a quick review of the basic concepts in Clifford analysis and generalized Fourier transforms.
Denoting by $\{e_{1}, e_{2}, \ldots, e_{m}\}$ the orthonormal basis of $\mathbb{R}^{m}$, the Clifford algebra $\mathcal{C}\ell_{0,m}$ over $\mathbb{R}^{m}$ is spanned by the reduced products \[\mathop{\cup}_{j=1}^{m}\{e_{\alpha}=e_{i_{1}}e_{i_{2}}\ldots e_{i_{j}}:\alpha=\{i_{1},i_{2},\ldots, i_{j}\}, \quad   1\le i_{1}<i_{2}<\cdots<i_{j}\le m\}\]  with the relations $e_{i}e_{j}+e_{j}e_{i}=-2\delta_{ij}$. We identify the point $x=(x_{1},\ldots, x_{m})$ in $\mathbb{R}^{m}$ with the vector variable $x=\sum_{j=1}^{m}e_{j}x_{j}$.  The inner product and the wedge product of two vectors $x,y\in \mathbb{R}^{m}$ can be defined by  the Clifford product: \[(x,y):=\sum_{j=1}^{m}x_{j}y_{j}=-\frac{1}{2}(xy+yx);\]
\[x\wedge y:=\sum_{j<k}e_{j}e_{k}(x_{j}y_{k}-x_{k}y_{j})=\frac{1}{2}(xy-yx).\]
We can find the Clifford product $xy=-(x,y)+x\wedge y$, and $(x\wedge y)^{2}=-|x|^{2}|y|^{2}+(x,y)^{2}$ (see \cite{HY}).
The complexified Clifford algebra $\mathcal{C}\ell_{0,m}^{c}$ is defined as $\mathbb{C}\otimes \mathcal{C}\ell_{0,m}$.

The conjugation is defined by $\overline{(e_{j_{1}}\ldots e_{j_{l}})}= (-1)^{l}e_{j_{l}}\ldots e_{j_{1}}$ as a linear mapping. For $x,y\in \mathcal{C}\ell_{0,m}^{c}$, we have $\overline{(xy)}=\overline{y}\overline{x}, \overline{\overline{x}}=x,$ and $\overline{i}=i$ which is not the usual complex conjugation. We define the Clifford norm of $x$ by $|x|^{2}=x\bar{x}, x\in \mathcal{C}\ell_{0,m}^{c}.$

The Dirac operator is given by $D=\sum_{j=1}^{m}e_{j}\partial_{x_{j}}.$ Together with the vector variable $x$, they satisfy the relations \[D^{2}=-\Delta, \qquad x^{2}=-|x|^{2}, \qquad \{x, D\}=-2\mathbb{E}-m ,\] where $\{a, b\}=ab+ba$ and $\mathbb{E}=\sum_{j=1}^{m}x_{j}\partial_{x_{j}}$ is the Euler operator and hence they generate a realization of the Lie superalgebra $\mathfrak{ osp}(1|2)$, which contains the Lie algebra $\mathfrak{sl}_{2}=\mbox{span}\{\Delta, |x|^{2}, [\Delta, |x|^{2}]\}$ as its even part.
 A function $u(x)$ is  called monogenic if $Du=0$. An important example of monogenic functions is the generalized Cauchy kernel
\[G(x)=\frac{1}{\omega_{m}}\frac{\bar{x}}{|x|^{m}}\] where $\omega_{m}$ is the surface area of the unit ball in $\mathbb{R}^{m}$. It is the fundamental solution of Dirac operator\cite{ca}. Note that the norm here is $|x|=(\sum_{i=1}^{m}x_{i}^{2})^{1/2}$ and coincides with Clifford norm.

Denote by $\mathcal{P}$ the space of polynomials taking values in $\mathcal{C}\ell_{0,m}$, i.e. $\mathcal{P}:=\mathbb{R}[x_{1},\ldots,x_{m}]\otimes\mathcal{C}\ell_{0,m}$. The space of homogeneous polynomials of degree $k$ is then denoted by $\mathcal{P}_{k}$. The space $\mathcal{M}_{k}:=(\mbox{ker}D)\cap \mathcal{P}_{k},$ is called the space of spherical monogenics of degree $k$.

The local behaviour of a monogenic function near a point can be investigated by the polynomials introduced above. The following theorem  is  the analogue of the Taylor series in complex analysis.
\begin{theorem} \label{ts}\cite{ca} Suppose $f$ is monogenic in an open set $\Omega$ containing the origin.  Then there exists an open neighbourhood $\Lambda$ of the origin in which $f$ can be developed into a normally convergent series of spherical monogenics $ M_{k}f(x)$, i.e.
\[f(x)=\sum_{k=0}^{\infty}M_{k}f(x),\]
with $M_{k}f(x)\in \mathcal{M}_{k}$.
\end{theorem}

 The classical Fourier transform
\[\mathcal{F}(f)(y)=(2\pi)^{-m/2}\int_{\mathbb{R}^{m}}e^{-i(x,y)}f(x)dx,\] with $(x,y)$ the usual inner product can be represented by  the operator exponential \cite{HR}, \cite{Rad1}  \[\mathcal{F}=e^{-i\frac{\pi}{4}(\Delta-|x|^{2}-m)}
.\]
The Clifford-Hermite functions  \[\psi_{2p,k,l}(x):=2^{p}p!L_{p}^{\frac{m}{2}+k-1}(|x|^{2})M_{k}^{l}e^{-|x|^{2}/2},\]
\[\psi_{2p+1,k,l}(x):=2^{p}p!L_{p}^{\frac{m}{2}+k}(|x|^{2})xM_{k}^{l}e^{-|x|^{2}/2},\]
where $p,k \in \mathbb{Z}_{\ge 0}$ and $\{M_{k}^{l}|l=1,\ldots, \dim(\mathcal{M}_{k})\}$ form a basis for $\mathcal{M}_{k}$, the space of spherical monogenics of degree $k$. They moreover realize the complete decomposition of the rapidly decreasing functions $\mathcal{S}(\mathbb{R}^{m})\otimes \mathcal{C}\ell_{m}\subset L^{2}(\mathbb{R}^{m})\otimes \mathcal{C}\ell_{m} $ in irreducible subspaces under the action of the dual pair $(Spin(m), \mathfrak{osp}(1|2))$.
The action of the regular Fourier transform on this basis is given by \begin{eqnarray}\label{s1}\mathcal{F}\psi_{j,k,l}=e^{-i\frac{\pi}{2}(j+k)}\psi_{j,k,l}=(-i)^{j+k}\psi_{j,k,l}.\end{eqnarray}

We further introduce the Gamma operator or the angular Dirac operator (see \cite{ca})
\begin{eqnarray}\label{ga}\Gamma_{x}:=-\sum_{j<k}e_{j}e_{k}(x_{j}\partial_{x_{k}}-x_{k}\partial_{x_{j}})=-x D_{x}-\mathbb{E}_{x},\end{eqnarray}
here $\mathbb{E}_{x}=\sum_{i=1}^{m}x_{i}\partial_{x_{i}}$ is the Euler operator. Note that $\Gamma_{x} $ commutes with scalar radial functions. The operator $\Gamma_{x}$ has two important eigenspaces:
\begin{eqnarray}\label{eg}\Gamma_{x}\mathcal{M}_{k}=-k\mathcal{M}_{k},\end{eqnarray}\begin{eqnarray}\label{eg1}\Gamma_{x}(x\mathcal{M}_{k-1})=(k+m-2)x\mathcal{M}_{k-1}\end{eqnarray}
which follows from the definition of $\Gamma_{x}$. The Scasimir $S$ in our operator realization of $\mathfrak{osp}(1|2)$ is related to the angular Dirac operator by $S=-\Gamma_{x}+\frac{m-1}{2}$, see \cite{lie}. The Casimir element $C=S^{2}$ acts on the Clifford-Hermite function by \[C\psi_{j,k,l}=(k+\frac{m-1}{2})^{2}\psi_{j,k,l}.\]
In \cite{Ro1}, the authors studied the full class of integral transforms which satisfy the following condition.
\begin{theorem} The properties

 (1) the Clifford-Helmholtz relations \[T\circ D_{x}=-iy\circ T,\]
\[T\circ x=-iD_{y}\circ T,\]

(2) $T\psi_{j,k,l}=\mu_{j,k}\psi_{j,k,l}$ with $\mu_{j,k}\in \mathbb{C},$

(3) $T^{4}=id$
are satisfied by the operators $T$ of the form \[T=e^{i\frac{\pi}{2}F(C)}e^{i\frac{\pi}{4}(\Delta-|x|^{2}-m)}\in e^{i\frac{\pi}{2}\bar{\mathcal{U}}(osp(1|2))}\]
where $F(C) $ is an operator that takes integer values when evaluated in the eigenvalues of $C$.
\end{theorem}
The integral kernel of the generalized Fourier transform $T$ can be expressed as $e^{i\frac{\pi}{2}F(C)}e^{-i(x,y)}$. We are in particular interested in the case where $F(C)$ reduces to a polynomial $G(\Gamma_{y})$ with integer coefficients.

\section{Generalized kernel in the Laplace domain}
\subsection{Closed expression for $e^{i\frac{\pi}{2}\Gamma_{y}^{2}}e^{-i(x,y)}$ }

In this subsection, we use the Laplace transform method to compute $e^{i\frac{\pi}{2}\Gamma_{y}^{2}}e^{-i(x,y)}$. The trick here will be used to compute the more general case in next subsection.
 We use the notation $\sqrt{+}:=\sqrt{s^{2}+|x|^{2}|y|^{2}}$. The following lemma was obtained in \cite{CDL}.
\begin{lemma} The Laplace transform of $t^{m/2-1}e^{-it(x,y)}$ can be expressed as
\begin{eqnarray}\label{le1} &&\mathcal{L}(t^{m/2-1}e^{-it(x,y)})=\frac{2^{m/2-1}\Gamma(m/2)}{\sqrt{+}(s+\sqrt{+})^{m/2-1}}
\frac{\displaystyle 1-\frac{iyx}{s+\sqrt{+}}+\frac{ iy(1-\frac{ \displaystyle iyx}{\displaystyle s+\sqrt{+}})x}{s+\sqrt{+}}}{\displaystyle \left|1-\frac{iyx}{s+\sqrt{+}}\right|^{m}}.\end{eqnarray}
\end{lemma}
In the following, we  will act with  $e^{i\frac{\pi}{2}\Gamma_{y}^{2}}$ on both sides of (\ref{le1}) to obtain the integral kernel in the Laplace domain.
Denote by \begin{eqnarray*} f(y)&=&\frac{2^{\frac{m}{2}}}{\sqrt{+}(s+\sqrt{+})^{m/2-1}} \frac{\displaystyle 1-\frac{iyx}{s+\sqrt{+}}}{\left|\displaystyle 1-\frac{iyx}{s+\sqrt{+}}\right|^{m}}\\&=&\frac{s+\sqrt{+}-iyx}{\sqrt{+}(s+i(x,y))^{m/2}},\end{eqnarray*}and \begin{eqnarray*}
g(y)&=&\frac{2^{\frac{m}{2}}}{\sqrt{+}(s+\sqrt{+})^{m/2-1}} \frac{\frac{ \displaystyle iy(1-\frac{ iyx}{\displaystyle s+\sqrt{+}})x}{\displaystyle s+\sqrt{+}}}{\displaystyle \left|1-\frac{iyx}{s+\sqrt{+}}\right|^{m}}\\&=&\frac{iy}{s+\sqrt{+}}f(y)x\\
&=&\frac{\sqrt{+}-s+iyx}{\sqrt{+}(s+i(x,y))^{m/2}}. \end{eqnarray*}
In \cite{CDL}, it has been proved that $f(y)$ has a series expansion as
\begin{eqnarray*}f(y)&=&\frac{2^{\frac{m}{2}}}{\sqrt{+}(s+\sqrt{+})^{m/2-1}}\sum_{k=0}^{\infty}\frac{M_{k}(y)}{(s+\sqrt{+})^{k}}
.\end{eqnarray*}
Here we rewrite
\[
f(y)=f_{0}(y)+f_{1}(y)+f_{2}(y)+f_{3}(y),
\]
with \begin{eqnarray}\label{fe1}f_{k}(y)=\frac{2^{\frac{m}{2}}}{\sqrt{+}(s+\sqrt{+})^{m/2-1}}\sum_{n=0}^{\infty}\frac{M_{4n+k}(y)}{(s+\sqrt{+})^{4n+k}}, \quad k=0, 1, 2,3.\end{eqnarray}
Each $f_{k}$ is  an eigenfunction of the operator $e^{i\frac{\pi}{2}\Gamma^{2}}.$   In fact, by (\ref{eg}), we have \begin{eqnarray*} e^{i\frac{\pi}{2}\Gamma_{y}^{2}} M_{k}(y)=e^{i\frac{\pi}{2}(-k)^{2}}M_{k}(y),
\end{eqnarray*} so \begin{eqnarray}\label{f1}
                           && e^{i\frac{\pi}{2}\Gamma_{y}^{2}} M_{4n}(y)=M_{4n}(y); \nonumber \\
                             &&e^{i\frac{\pi}{2}\Gamma_{y}^{2}} M_{4n+1}(y)= i M_{4n+1}(y); \nonumber\\
                             &&e^{i\frac{\pi}{2}\Gamma_{y}^{2}} M_{4n+2}(y)= M_{4n+2}(y); \nonumber\\
                             &&e^{i\frac{\pi}{2}\Gamma_{y}^{2}} M_{4n+3}(y)=i M_{4n+3}(y),
                         \end{eqnarray}
here $n=0,1,2,\cdots.$ Since the operator $\Gamma$ commutes with radial functions, we know that each $f_{k}$ is an eigenfunction of $e^{i\frac{\pi}{2}\Gamma^{2}}$ and the eigenvalues are given in (\ref{f1}).
In the following, we denote
\begin{eqnarray*}
f_{\alpha}(y)&=&\frac{2^{\frac{m}{2}}}{\sqrt{+}(s+\sqrt{+})^{m/2-1}}\sum_{k=0}^{\infty}\frac{M_{k}(iy)}{(s+\sqrt{+})^{k}}\\&=&\frac{s+\sqrt{+}+yx}{\sqrt{+}(\sqrt{+}-(x,y))^{m/2}},\end{eqnarray*}
\begin{eqnarray*}
f_{\beta}(y)&=&\frac{2^{\frac{m}{2}}}{\sqrt{+}(s+\sqrt{+})^{m/2-1}}\sum_{k=0}^{\infty}\frac{M_{k}(-y)}{(s+\sqrt{+})^{k}}\\&=&\frac{s+\sqrt{+}+iyx}{\sqrt{+}(s-i(x,y))^{m/2}},
\end{eqnarray*}
\begin{eqnarray*}
f_{\gamma}(y)&=&\frac{2^{\frac{m}{2}}}{\sqrt{+}(s+\sqrt{+})^{m/2-1}}\sum_{k=0}^{\infty}\frac{M_{k}(-iy)}{(s+\sqrt{+})^{k}}\\&=&\frac{s+\sqrt{+}-yx}{\sqrt{+}(\sqrt{+}+(x,y))^{m/2}}
\end{eqnarray*}
as well as
\begin{eqnarray*}&&g_{\alpha}(y)=\frac{iy}{s+\sqrt{+}}f_{\alpha}(y)x=\frac{i(\sqrt{+}-s)+iyx}{\sqrt{+}(\sqrt{+}-(x,y))^{m/2}},
\\&&g_{\beta}(y)=\frac{iy}{s+\sqrt{+}}f_{\beta}(y)x=\frac{s-\sqrt{+}+iyx}{\sqrt{+}(s-i(x,y))^{m/2}},
\\&& g_{\gamma}(y)=\frac{iy}{s+\sqrt{+}}f_{\gamma}(y)x=\frac{i(s-\sqrt{+})+iyx}{\sqrt{+}(\sqrt{+}+(x,y))^{m/2}}.
\end{eqnarray*}
\begin{remark}\label{rma1}
Comparing with Theorem 3 in \cite{CDL}, $\frac{\Gamma(m/2)}{2}(f_{\gamma}+g_{\alpha})$ is the Clifford-Fourier  kernel of  dimension $m=4n+1, n\in \mathbb{N}$ in the Laplace domain. Denote the first part of the fractional Clifford-Fourier kernel as
 \[
 F_{p}(x,y)=\frac{s+\sqrt{+}-ie^{-ip}yx}{\sqrt{+}(e^{-ip}(s\cos p+i\sqrt{+}\sin p+i(x,y)))^{m/2}}
\]
and the second part of the kernel as
\[
G_{p}(x,y)=-e^{ip}\frac{s-\sqrt{+}-ie^{-ip}yx}{\sqrt{+}(e^{ip}(s\cos p-i\sqrt{+}\sin p+i(x,y)))^{m/2}}.
\]
We find that
 $f(y)=F_{0}(x,y)$, $f_{\alpha}(y)=F_{-\frac{\pi}{2}}(x,y)$, $f_{\beta}(y)=F_{\pi}(x,y)$, $f_{\gamma}(y)=F_{\frac{\pi}{2}}(x,y)$,
 $g(y)=G_{0}(x,y)$, $g_{\alpha}(y)=G_{\frac{\pi}{2}}(x,y)$, $g_{\beta}(y)=G_{\pi}(x,y)$ and $g_{\gamma}(y)=G_{-\frac{\pi}{2}}(x,y)$.
 We could get the plane wave expansion and integral expression of $f, f_{\alpha,\beta,\gamma}$ and $g, g_{\alpha,\beta,\gamma}$ from \cite{CDL}.
\end{remark}
As $M_{k}$ is a polynomial of degree $k$, we have the following relations, \[ \left\{
              \begin{array}{ll}
                 f(y)=f_{0}(y)+f_{1}(y)+f_{2}(y)+f_{3}(y);\\
                 f_{\alpha}(y)=f_{0}(y)+if_{1}(y)-f_{2}(y)-if_{3}(y);\\
                 f_{\beta}(y)=f_{0}(y)-f_{1}(y)+f_{2}(y)-f_{3}(y);\\
                f_{\gamma}(y)=f_{0}(y)-if_{1}(y)-f_{2}(y)+if_{3}(y).
              \end{array}
            \right.\]  Each $f_{k}(y)$ can be obtained as following:
\begin{eqnarray}\label{ne1} \left\{
              \begin{array}{ll}
                 4f_{0}(y)=f(y)+f_{\alpha}(y)+f_{\beta}(y)+f_{\gamma}(y);\\
                 4f_{1}(y)=f(y)-if_{\alpha}(y)-f_{\beta}(y)+if_{\gamma}(y);\\
                 4f_{2}(y)=f(y)-f_{\alpha}(y)+f_{\beta}(y)-f_{\gamma}(y);\\
                4f_{3}(y)=f(y)+if_{\alpha}(y)-f_{\beta}(y)-if_{\gamma}(y).
              \end{array}
            \right.\end{eqnarray}
Now the action of  $e^{i\frac{\pi}{2}\Gamma_{y}^{2}}$ on $f(y)$ is known through its eigenfunctions,
\begin{eqnarray*}e^{i\frac{\pi}{2}\Gamma_{y}^{2}} f(y)&=&e^{i\frac{\pi}{2}\Gamma_{y}^{2}}\biggl(f_{0}(y)+f_{1}(y)+f_{2}(y)+f_{3}(y)\biggr)\\&=& f_{0}(y)+if_{1}(y)+f_{2}(y)+if_{3}(y)\\&=& \frac{1}{2}\biggl( f(y)+f_{\beta}(y)+if(y)-if_{\beta}(y)\biggr).
\end{eqnarray*}
The case $e^{i\frac{\pi}{2}\Gamma_{y}^{2}}g(y)$ can be treated similarly,
using (\ref{f1}) and \begin{eqnarray*}e^{i\frac{\pi}{2}\Gamma_{y}^{2}}(yM_{k}(y))&=&e^{i\frac{\pi}{2}(m-1+k)^{2}}(yM_{k}(y))\\&=&e^{i\frac{\pi}{2}(m-1)^{2}}e^{i\frac{\pi}{2}k^{2}}(yM_{k}(e^{i\pi(m-1)}y))\\&=&e^{i\frac{\pi}{2}(m-1)^{2}}y e^{i\frac{\pi}{2}k^{2}}(M_{k}(e^{i\pi(m-1)}y)).\end{eqnarray*}
 Collecting everything, we have
\begin{theorem}\label{the2} The kernel $t^{m/2-1}e^{i\frac{\pi}{2}\Gamma_{y}^{2}}e^{-i(x,y)}$ in the Laplace domain is \begin{eqnarray*}&&\mathcal{L}(t^{m/2-1}e^{i\frac{\pi}{2}\Gamma_{y}^{2}}e^{-it(x,y)})\\&=&\frac{\Gamma(m/2)}{4\sqrt{+}}\biggl((1+i)U^{1}_{m}+(1-i)U^{2}_{m}+e^{i\frac{\pi}{2}(m-1)^{2}}((1+i)U^{3}_{m}+(1-i)U^{4}_{m}) \biggr), \end{eqnarray*}
with \\
$U^{1}_{m}=\frac{\displaystyle s+\sqrt{+}-iyx}{\displaystyle (s+i(x,y))^{m/2}};$ \qquad
$U^{2}_{m}=\frac{\displaystyle s+\sqrt{+}+iyx}{\displaystyle (s-i(x,y))^{m/2}};\\$
$U^{3}_{m}=\frac{\displaystyle (-1)^{m-1}(\sqrt{+}-s)+iyx}{\displaystyle (s+(-1)^{m-1}i(x,y))^{m/2}};$ \qquad
$U^{4}_{m}=\frac{\displaystyle (-1)^{m-1}(s-\sqrt{+})+iyx}{\displaystyle (s-(-1)^{m-1}i(x,y))^{m/2}},\\$
where $\sqrt{+}=\sqrt{s^{2}+|x|^{2}|y|^{2}}$.
\end{theorem}
When $m=2$, \begin{eqnarray*}&&\mathcal{L}(e^{i\frac{\pi}{2}\Gamma_{y}^{2}}e^{-it(x,y)})\\
&=&\frac{1}{2\sqrt{+}}\biggl(\frac{\sqrt{+}}{s-i(x,y)}+ \frac{s-iyx}{s+i(x,y)} \biggr).
\end{eqnarray*}
By formula (\ref{l7}), (\ref{l2}), and the convolution formula (\ref{cv1}),  the kernel equals, putting $t=1$, \begin{eqnarray*}K_{2, \Gamma^{2}}(x,y)=e^{i(x,y)}+J_{0}(|x||y|)+ix\wedge y\int_{0}^{1}e^{-i(x,y)(1-\tau)}J_{0}(|x||y|\tau)d\tau.\end{eqnarray*}

In the following, we analyze each term in Theorem \ref{the2} in detail. By formula (\ref{l8}), (\ref{cv1}) and (\ref{l2}), letting $t=1$, we get  $U^{1}_{m},  U^{2}_{m},  U^{3}_{m},  U^{4}_{m}$  in the time domain as
\begin{eqnarray*}
K_{U^{1}_{m}}&=&\frac{e^{-i(x,y)}}{\Gamma(m/2)}+\frac{1}{\Gamma(m/2-1)}\int_{0}^{1}\tau^{m/2-2}e^{-i(x,y)\tau}J_{0}(|x||y|(1-\tau))d\tau
\\&&+\frac{ix\wedge y}{\Gamma(m/2)}\int_{0}^{1}e^{-i(x,y)}J_{0}(|x||y|(1-\tau))d\tau,
\\
K_{U^{2}_{m}}&=&\frac{e^{i(x,y)}}{\Gamma(m/2)}+\frac{1}{\Gamma(m/2-1)}\int_{0}^{1}\tau^{m/2-2}e^{i(x,y)\tau}J_{0}(|x||y|(1-\tau))d\tau\\&&-\frac{ix\wedge y}{\Gamma(m/2)}\int_{0}^{1}e^{i(x,y)}J_{0}(|x||y|(1-\tau))d\tau,
\\
K_{U^{3}_{m}}&=&(-1)^{m-1}(\frac{1}{\Gamma(m/2)}e^{i(-1)^{m}(x,y)}\\&&-\frac{1}{\Gamma(m/2-1)}\int_{0}^{1}\tau^{m/2-2}e^{i(-1)^{m}(x,y)\tau}J_{0}(|x||y|(1-\tau))d\tau)\\&&-\frac{ix\wedge y}{\Gamma(m/2)}\int_{0}^{1}e^{i(-1)^{m}(x,y)}J_{0}(|x||y|(1-\tau))d\tau,
\end{eqnarray*}\begin{eqnarray*}
K_{U^{4}_{m}}&=&(-1)^{m-1}(-\frac{1}{\Gamma(m/2)}e^{i(-1)^{m-1}(x,y)}\\&&+\frac{1}{\Gamma(m/2-1)}\int_{0}^{1}\tau^{m/2-2}e^{i(-1)^{m-1}(x,y)\tau}J_{0}(|x||y|(1-\tau))d\tau)\\&&-\frac{ix\wedge y}{\Gamma(m/2)}\int_{0}^{1}e^{i(-1)^{m-1}(x,y)}J_{0}(|x||y|(1-\tau))d\tau
.\end{eqnarray*}

\begin{theorem}\label{th4} Let $m\ge 2$. For $x,y \in \mathbb{R}^{m}$, the generalized Fourier kernel
is given by
\begin{eqnarray*} K_{m,\Gamma^2}(x,y)&=&\frac{\Gamma(m/2)}{4}\biggl((1+i)K_{U^{1}_{m}}+(1-i)K_{U^{2}_{m}}\\&&+e^{i\frac{\pi}{2}(m-1)^{2}}((1+i)K_{U^{3}_{m}}+(1-i)K_{U^{4}_{m}}) \biggr).
\end{eqnarray*}
There exists a constant $c$ such that
\[
|K_{m,\Gamma^2}(x,y)|\le c(1+|x||y|).
\]
\end{theorem}
\begin{proof} This follows from the fact that $J_{0}(y)$ and $e^{i(x,y)}$ are bounded functions and $|x\wedge y|\le |x||y|$.
\end{proof}
\subsection{Closed expression for $e^{i\frac{\pi}{2}G(\Gamma_{y})}e^{-i(x,y)}$}
In this subsection, we consider the more general case.  We act with $G(\Gamma_{y})$ on the Fourier kernel. Here $G(x)$ is a polynomial with integer coefficients,
\[G(x)=a_{n}x^{n}+a_{n-1}x^{n-1}+\cdots+a_{1}x+a_{0},\qquad a_{k}\in \mathbb{Z}.\]
Using the fact that  $e^{i\frac{\pi}{2}j}$ is 4-periodic in $j$,
\[
e^{i\frac{\pi}{2}G(\Gamma_{y})}M_{k}(y)=e^{i\frac{\pi}{2}G(-k)}M_{k}(y)
\]
 and
  \[
  G(4n+k)\equiv G(k)(\mbox{mod} 4),\]
 we have
\begin{eqnarray*}e^{i\frac{\pi}{2}G(\Gamma_{y})}f(y)&=&e^{i\frac{\pi}{2}G(0)}f_{0}+e^{i\frac{\pi}{2}G(-1)}f_{1}+e^{i\frac{\pi}{2}G(-2)}f_{2}+e^{i\frac{\pi}{2}G(-3)}f_{3}
\\&=&i^{G(0)}f_{0}+i^{G(-1)}f_{1}+i^{G(-2)}f_{2}+i^{G(-3)}f_{3},\end{eqnarray*}
with each $f_{k}$ defined in (\ref{fe1}).
 By
 \[
 e^{i\frac{\pi}{2}G(\Gamma_{y})}(yM_{k}(y))=e^{i\frac{\pi}{2}G(m-1+k)}(yM_{k})
 \]
and \[G(4n+k+m-1)\equiv G(k+m-1)(\mbox{mod} 4),\]
we have
 \begin{eqnarray*}&&e^{i\frac{\pi}{2}G(\Gamma_{y})}g(y)\\&=&\frac{iy}{s+\sqrt{+}}\biggl(e^{i\frac{\pi}{2}G(m-1)}f_{0}+e^{i\frac{\pi}{2}G(m)}f_{1}+e^{i\frac{\pi}{2}G(m+1)}f_{2}+e^{i\frac{\pi}{2}G(m+2)}f_{3}\biggr)x\\
 &=&\frac{iy}{s+\sqrt{+}}\biggl(i^{G(m-1)}f_{0}+i^{G(m)}f_{1}+i^{G(m+1)}f_{2}+i^{G(m+2)}f_{3}\biggr)x.\end{eqnarray*}
 Collecting everything we obtain and applying (\ref{ne1}), we get
\begin{theorem}\label{th5} For $G(x)\in \mathbb{Z}[x]$, the Laplace transform of $t^{m/2-1}e^{i\frac{\pi}{2}G(\Gamma_{y})}e^{-it(x,y)}$ is given by  \begin{eqnarray*} &&\mathcal{L}(t^{m/2-1}e^{i\frac{\pi}{2}G(\Gamma_{y})}e^{-it(x,y)})\\&=&\frac{\Gamma(m/2)}{8}\biggl(A_{m}^{1}BC^{T}_{m} +\frac{iy}{s+\sqrt{+}}A_{m}^{2}BC^{T}_{m}x \biggr)\end{eqnarray*}
with $A_{m}^{1}, A_{m}^{2}, B,  C_{m}$ the matrices given by \begin{eqnarray*}
&&A_{m}^{1}=\begin{pmatrix}i^{G(0)}& i^{G(-1)}& i^{G(-2)}& i^{G(-3)}\end{pmatrix},\\
&&A_{m}^{2}=\begin{pmatrix}i^{G(m-1)}& i^{G(m)}& i^{G(m+1)}& i^{G(m+2)}\end{pmatrix},\\
&&B=\begin{pmatrix}1&1&1&1\\1&-i&-1&i\\1&-1&1&-1\\1&i&-1&-i\end{pmatrix},\\&&C_{m}=\begin{pmatrix}f(y)&f_{\alpha}(y)&f_{\beta}(y)&f_{\gamma}(y)\end{pmatrix}.\end{eqnarray*}
\end{theorem}
\begin{remark} We could get the regular Fourier kernel $e^{-i(x,y)}$ by setting $G(x)=0$ or $4x$ for dimension $m\ge 2$. When $G=2x^{2}$, we get the inverse Fourier kernel $e^{i(x,y)}$ for even dimension. When $G(x)=\pm x$, it is the Clifford-Fourier transform \cite{HY}.
\end{remark}
As the constant term of the polynomial will only contribute a constant factor to the integral kernel, in the following we only consider  polynomials without constant term
\[G(x)=a_{n}x^{n}+a_{n-1}x^{n-1}+\cdots+a_{1}x, \qquad a_{k}\in \mathbb{Z}.\]
By  \[
  G(4n+k)\equiv G(k)(\mbox{mod} 4),\]
it reduces to four cases $G(k)(\mbox{mod} 4),$  $k=0, 1, 2, 3.$  The set $\{x^{m}\}\cup \{1\}, m\in \mathbb{N}$ is a basis for  polynomials over the ring of integers. We consider the four cases on this basis
\begin{eqnarray*}
&&x^{j}= 0, \qquad \mbox{when} \quad  x=0;\\
&&x^{j} = 1, \qquad \mbox{when} \quad  x=1;\\
&&x^{j}\equiv \left\{
               \begin{array}{ll}
                 2(\mbox{mod} 4), & \hbox{when $j=1$ and $x=2$}; \\
                 0(\mbox{mod} 4), & \hbox{when $j\ge 2$ and $x=2$};
               \end{array}
             \right. \\
&&x^{j} \equiv \left\{
               \begin{array}{ll}
                 1(\mbox{mod} 4), & \hbox{when $j$ is even and $x=3$}; \\
                 3(\mbox{mod} 4), & \hbox{when $j$ is odd and $x=3$}.
               \end{array}
             \right.
\end{eqnarray*} For each $G(x)$,  we denote $\frac{G(1)+G(-1)}{2}=s_{0}= \sum_{j=0}^{\lfloor n/2\rfloor} a_{2j}$  and $\frac{G(1)-G(-1)}{2}=s_{1}= \sum_{j=0}^{\lfloor n/2\rfloor} a_{2j+1}$ with $n$ the degree of $G(x)$. We have
\begin{eqnarray*} &&G(0)=0,\\
&& G(1)=s_{0}+s_{1},\\
&&G(2)\equiv 2a_{1}(\mbox{mod} 4),\\
&& G(3) \equiv G(-1)\equiv s_{0}-s_{1}(\mbox{mod} 4).
\end{eqnarray*}
Therefore \begin{eqnarray*}&&i^{G(0)}=1,\qquad i^{G(-1)}=i^{G(3)}=i^{s_{0}+3s_{1}},\\&& i^{G(-2)}=i^{G(2)}=(-1)^{a_{1}}, \qquad i^{G(-3)}=i^{G(1)}=i^{s_{0}+s_{1}}.\end{eqnarray*}

The class of integral transforms with polynomially bounded kernel is of great interest. For example, new uncertainty principles have been given for this kind of integral transforms in \cite{GJ}. As we can see in Theorem \ref{th5}, the generalized Fourier kernel is a linear combination of $f_{\alpha,\beta,\gamma}, f, g_{\alpha,\beta,\gamma},  g $. At present, very few of $f_{\alpha}, f_{\gamma}, g_{\alpha}, g_{\gamma}$ are known explicitly.  The integral representations of $f_{\alpha}, f_{\gamma}, g_{\alpha}, g_{\gamma}$ are obtained in \cite{CDL} but without the bound. Only in  even dimensions, special linear combinations of $f_{\alpha}, f_{\gamma}, g_{\alpha}, g_{\gamma}$ are known to be polynomially bounded which is exactly the Clifford-Fourier kernel \cite{HY}.

We have showed  in Theorem \ref{th4}  that $f, f_{\beta}, g, g_{\beta}$ with polynomial bounds behaves better than $f_{\alpha}, f_{\gamma}, g_{\alpha}, g_{\gamma}$. So it is interesting to consider the generalized Fourier transform whose kernel only consists of  $f, f_{\beta}, g, g_{\beta}$. It also provides ways to define hypercomplex Fourier transforms with polynomially bounded kernel in odd dimensions.  We will hence characterize  polynomials such that   $e^{i\frac{\pi}{2}G(\Gamma_{y})}e^{-i(x,y)}$ are only linear combination of $f, f_{\beta}, g, g_{\beta}$.

 For fixed $m$, the kernel is a linear sum of $f, f_{\beta}, g, g_{\beta}$  when  the polynomial $G(x)\in \mathbb{Z}[x]$ satisfies the following conditions
\begin{equation}\label{4f}
\left\{
  \begin{array}{ll}
    i^{G(0)}-ii^{G(-1)}-i^{G(-2)}+ii^{G(-3)}=0, \\
    i^{G(0)}+ii^{G(-1)}-i^{G(-2)}-ii^{G(-3)}=0, \\
    i^{G(m-1)}-ii^{G(m)}-i^{G(m+1)}+ii^{G(m+2)}=0,\\
    i^{G(m-1)}+ii^{G(m)}-ii^{G(m+1)}-ii^{G(m+2)}=0.
  \end{array}
\right.
\end{equation}
We find that (\ref{4f}) is equivalent with
\begin{equation}\label{cl1}
\left\{
  \begin{array}{ll}
   G(0)\equiv G(-2)(\mbox{mod} 4),\\
   G(-1)\equiv G(-3)(\mbox{mod} 4),\\
   G(m-1)\equiv G(m+1)(\mbox{mod} 4),\\
   G(m)\equiv G(m+2)(\mbox{mod} 4).\\
  \end{array}
\right.
\end{equation}
As $G(k)(\mbox{mod} 4)$ is uniquely determined by $G(0), G(-1), G(-2)$ and $G(-3)$, the first two formulas in $(\ref{cl1})$  imply the last two formulas for all $m\ge 2$ automatically. Now $(\ref{cl1})$  becomes
\begin{eqnarray*}\left\{
                                                                                  \begin{array}{ll}
                                                                                    i^{G(0)}=1=i^{G(-2)}=(-1)^{a_{1}},  \\
                                                                                     i^{G(-1)}=i^{s_{0}+3s_{1}}=i^{G(-3)}=i^{s_{0}+s_{1}}.
                                                                                  \end{array}
                                                                                \right.
 \end{eqnarray*} It follows that the kernel  only consists of $f, f_{\beta}, g, g_{\beta}$  if and only if $a_{1}$ and $s_{1}$ are even.
We have the following
\begin{theorem} Let $m\ge 2$. For $x,y \in \mathbb{R}^{m}$ and a polynomial $G(x)$ with integer coefficients, the kernel $e^{i\frac{\pi}{2}G(\Gamma_{y})}e^{-i(x,y)}$ is a linear combination of $f, f_{\beta},g, g_{\beta}$ in the Laplace domain if and only if $a_{1}$ and $\frac{G(1)-G(-1)}{2}$ are even. Furthermore, the generalized Fourier kernel is bounded and equals
\[
\frac{1+i^{G(1)}}{2}e^{-i(x,y)}+\frac{1-i^{G(1)}}{2}K^{\pi}(x,y),
\]
with $K^{\pi}(x,y)$ the fractional Clifford-Fourier kernel in \cite{CDL}.
When $m\ge 2$ is even, the kernel is
\[
\frac{1+i^{G(1)}}{2}e^{-i(x,y)}+\frac{1-i^{G(1)}}{2}e^{i(x,y)}.
\]
When $m\ge 2$ is odd,  there exists a constant $c$ which is independent of $m$ such
that
\begin{eqnarray}\label{ne2}
|e^{i\frac{\pi}{2}G(\Gamma_{y})}e^{-i(x,y)}|\le c(1+|x||y|).
\end{eqnarray}
\end{theorem}
\begin{proof} We only need to prove the generalized Fourier kernel is
\[
\frac{1+i^{s_{0}+s_{1}}}{2}e^{-i(x,y)}+\frac{1-i^{s_{0}+s_{1}}}{2}K^{\pi}(x,y).
\]
In fact, by verification, we have,
\begin{eqnarray*} (e^{i0})^{m-1}A_{m}^{1}\begin{pmatrix}1\\1\\1\\1\end{pmatrix}=A_{m}^{2}\begin{pmatrix}1\\1\\1\\1\end{pmatrix};\qquad
(e^{i\pi})^{m-1}A_{m}^{1}\begin{pmatrix}1\\-1\\1\\-1\end{pmatrix}=A_{m}^{2}\begin{pmatrix}1\\-1\\1\\-1\end{pmatrix},
\end{eqnarray*}
and
\begin{eqnarray*}  A_{m}^{1}\begin{pmatrix}1\\1\\1\\1\end{pmatrix}=2+2i^{s_{0}+s_{1}};\qquad A_{m}^{1}\begin{pmatrix}1\\-1\\1\\-1\end{pmatrix}=2-2i^{s_{0}+s_{1}}.
\end{eqnarray*}
By Remark \ref{rma1}, $f+(e^{i0})^{m-1}g$ is the kernel $K_{0}$ and  $f_{\beta}+(e^{i\pi})^{m-1}g_{\beta}$ is the fractional Clifford-Fourier kernel $K^{\pi}$. The bound (\ref{ne2}) follows from the integral expression of  $f, f_{\beta},g, g_{\beta}$ in the time domain.
\end{proof}
\begin{remark} The case $G(x)=x^2$ is a special case of this theorem.
\end{remark}
In the following, we consider the generalized Fourier kernel which has polynomial bound and consists of $f_{\alpha,\beta,\gamma}, f, g_{\alpha,\beta,\gamma},  g $.
For even dimension,  we already know the Clifford-Fourier kernel has a polynomial bound. If the polynomial $G(x)$ satisfies
 \begin{eqnarray}\label{p1} (-i)^{m-1}A_{m}^{1}\begin{pmatrix}1\\-i\\-1\\i\end{pmatrix}=A_{m}^{2}\begin{pmatrix}1\\i\\-1\\-i\end{pmatrix};\qquad
i^{m-1}A_{m}^{1}\begin{pmatrix}1\\i\\-1\\-i\end{pmatrix}=A_{m}^{2}\begin{pmatrix}1\\-i\\-1\\i\end{pmatrix},
\end{eqnarray}
by Remark \ref{rma1}, $e^{i\frac{\pi}{2}G(\Gamma_{y})}e^{-i(x,y)}$ is a linear combination of the Clifford-Fourier kernel and some function bounded by $c(1+|x||y|)$. Hence it has a polynomial bound as well. When $m=4j,$ (\ref{p1}) becomes
\[i(1-i^{s_{0}+3s_{1}+1}-(-1)^{a_{1}}+i^{s_{0}+s_{1}+1})=i^{s_{0}+3s_{1}}+i-i^{s_{0}+s_{1}}-i(-1)^{a_{1}}\]
and
\[-i(1+i^{s_{0}+3s_{1}+1}-(-1)^{a_{1}}-i^{s_{0}+s_{1}+1})=i^{s_{0}+3s_{1}}-i-i^{s_{0}+s_{1}}+i(-1)^{a_{1}}.\]
It shows that  (\ref{p1}) is true for any $G(x)\in \mathbb{Z}[x]$ when $m=4j$. When $m=4j+2$, (\ref{p1}) becomes
\[-i(1-i^{s_{0}+3s_{1}+1}-(-1)^{a_{1}}+i^{s_{0}+s_{1}+1})=i^{s_{0}+s_{1}}+i(-1)^{a_{1}}- i^{s_{0}+3s_{1}}-i\]
and
\[i(1+i^{s_{0}+3s_{1}+1}-(-1)^{a_{1}}-i^{s_{0}+s_{1}+1})=i^{s_{0}+s_{1}}-i(-1)^{a_{1}}-i^{s_{0}+3s_{1}}+i.\]
It also shows that  (\ref{p1}) is true for any $G(x)\in \mathbb{Z}[x]$ when $m=4j+2$.
Now we have
\begin{theorem} Let $m\ge 2$ be even. For $x,y \in \mathbb{R}^{m}$ and any polynomial $G(x)$ with integer coefficients, the kernel $e^{i\frac{\pi}{2}G(\Gamma_{y})}e^{-i(x,y)}$ has a polynomial bound, i.e. there exists a constant $c$  which is independent of $G(x)$ such that
\[
|e^{i\frac{\pi}{2}G(\Gamma_{y})}e^{-i(x,y)}| \le c(1+|x||y|)^{\frac{m-2}{2}}.
\]
\end{theorem}
At the end of this section, we give the formal generating function of the even dimensional generalized Fourier kernels for a class of polynomials. We define
\[H(x,y,a, G)= \sum_{m=2,4,6,\cdots}\frac{K_{m,G}(x,y)a^{m/2-1}}{\Gamma(m/2)}.
\]
\begin{theorem} Let $m\ge 2$ be even. For $x,y \in \mathbb{R}^{m}$ and any polynomial $G(x)$ with integer coefficients,  the formal generating function of the even dimensional generalized Fourier kernel is given by
 \begin{eqnarray*}
 &&H(x,y,a,G) \\&=&\frac{1-i^{G(-1)+1}-(-1)^{G'(0)}+i^{G(1)+1}}{2}\biggl(\cos (\sqrt{|x|^2|y|^2-((x,y)+a)^2})-(x\wedge y-a) \frac{\sin\sqrt{|x|^2|y|^2-((x,y)+a)^2}}{\sqrt{|x|^2|y|^2-((x,y)+a)^2}}\biggr)\\&&+\frac{1+i^{G(-1)+1}-(-1)^{G'(0)}-i^{G(1)+1}}{2}\biggl(
 \cos(\sqrt{|x|^{2}|y|^{2}-((x,y)-a)^{2}})+(x\wedge y+a)\frac{\sin\sqrt{|x|^{2}|y|^{2}-((x,y)-a)^{2}}}{\sqrt{|x|^{2}|y|^{2}-((x,y)-a)^{2}}}\biggr)\\&&+\frac{1+i^{G(-1)}+(-1)^{G'(0)}+i^{G(1)}}{2}e^{-(i(x,y)-a)}
 +\frac{1-i^{G(-1)}+(-1)^{G'(0)}-i^{G(1)}}{2}e^{i(x,y)+a}.
 \end{eqnarray*}

\end{theorem}
\begin{proof} When $m$ is even, the generalized Fourier kernel is
\begin{eqnarray*}e^{i\frac{\pi}{2}G(\Gamma_{y})}e^{-i(x,y)}&=&\frac{1}{2}\biggl( (1-i^{s_{0}+3s_{1}+1}-(-1)^{a_{1}}+i^{s_{0}+s_{1}+1})(f_{\alpha}+e^{i\frac{-\pi}{2}(m-1)}g_{\gamma} )\\&&+(1+i^{s_{0}+3s_{1}+1}-(-1)^{a_{1}}-i^{s_{0}+s_{1}+1})(f_{\gamma}+e^{i\frac{\pi}{2}(m-1)}g_{\alpha})\\&&+(1+i^{s_{0}+3s_{1}}+(-1)^{a_{1}}+i^{s_{0}+s_{1}})e^{-i(x,y)}+ (1-i^{s_{0}+3s_{1}}+(-1)^{a_{1}}-i^{s_{0}+s_{1}})e^{i(x,y)}\biggr),\end{eqnarray*}
 with $s_{0}= \sum_{j=0}^{\lfloor n/2\rfloor} a_{2j}$  and $s_{1}= \sum_{j=0}^{\lfloor n/2\rfloor} a_{2j+1}$.

By $s_{0}+3s_{1}\equiv s_{0}-s_{1}\equiv G(-1) (\mbox{mod} 4)$, $s_{0}+s_{1}=G(1)$,  $a_{1}=G'(0)$ as well as because
$\frac{\Gamma(m/2)}{2}(f_{\alpha}+e^{i\frac{-\pi}{2}(m-1)}g_{\gamma}) $ and $\frac{\Gamma(m/2)}{2}(f_{\gamma}+e^{i\frac{\pi}{2}(m-1)}g_{\alpha})$ are the Clifford-Fourier kernel $K^{\frac{-\pi}{2}}$ and $K^{\frac{\pi}{2}}$ in the Laplace domain, the results follows from the generating function of Clifford-Fourier kernel, see \cite{CDL} Theorem 8.
\end{proof}
\begin{remark} When $G(x)=x$, we get the generating function of the Clifford-Fourier kernel.
\end{remark}
 For the case that the coefficients of $G(x)$ are not integers but fractions, we write $G_{1}(x)=cG(x)$ in which $c$ is the least common multiple of each denominator of $G(x)$. So $G_{1}(x)$ is a polynomial with integer coefficients. We only need to compute $e^{i\frac{\pi}{2c}G_{1}(\Gamma_{y})}f(y)$ and $e^{i\frac{\pi}{2c}G_{1}(\Gamma_{y})}g(y)$. The same method will also work but $f$ and $g$ split into $4c$ parts.

 \

 \footnotesize{\bf Acknowledgments} \qquad H. De Bie is supported by the UGent BOF starting grant 01N01513. P. Lian is supported by the scholarship from Chinese Scholarship Council (CSC).

\end{document}